\documentclass[10pt,psamsfonts]{amsart}
\usepackage{amsmath}
\usepackage{amsthm}
\usepackage{amssymb}
\usepackage{amscd}
\usepackage{amsfonts}
\usepackage{amsbsy}
\usepackage{graphicx}
\usepackage{url}
\usepackage{overpic}
\usepackage{hyperref}
\usepackage{float}

\newcommand{\R}{\ensuremath{\mathbb{R}}}
\newcommand{\N}{\ensuremath{\mathbb{N}}}

\newcommand{\ov}{\overline}

\newcommand{\T}{\theta}
\newcommand{\f}{\varphi}

\newcommand{\dom}{\mathrm{dom~}}
\renewcommand{\Im}{{\rm Im~}}
\renewcommand{\ker}{{\rm ker~}}
\newcommand{\coker}{{\rm coker~}}

\newcommand{\sgn}{\mathrm{sign}}

\newcommand{\inte}{\mathrm{int}}
\def\p{\partial}
\def\e{\varepsilon}

\newtheorem {theorem} {Theorem}

\newtheorem {proposition}{Proposition}

\newtheorem {example} {Example}
\newtheorem {remark}{Remark}

\newtheorem {mtheorem} {Theorem}

\usepackage{mathpazo}

\begin{document}

\title[Averaging method for Carath\'{e}odory differential equations]
{An averaging result for periodic solutions of\\ Carath\'{e}odory differential equations}

\author[D.D. Novaes]
{Douglas D. Novaes}

\address{Departamento de Matem\'{a}tica - Instituto de Matem\'{a}tica, Estat\'{i}stica e Computa\c{c}\~{a}o Cient\'{i}fica (IMECC) - Universidade Estadual de Campinas (UNICAMP), Rua S\'{e}rgio Buarque de Holanda, 651, Cidade
Universit\'{a}ria Zeferino Vaz, 13083--859, Campinas, SP, Brazil}
\email{ddnovaes@unicamp.br}



\begin{abstract}
This paper is concerned with the problem of existence of periodic solutions for perturbative Carath\'{e}odory differential equations. The main result provides sufficient conditions on the averaged equation that guarantee the existence of periodic solutions. Additional conditions are also provided to ensure the uniform convergence of a periodic solution to a constant function. The proof of the main theorem is mainly based on an abstract continuation result for operator equations.
\end{abstract}

\maketitle


\section{Introduction and statement of the main result}

This paper is concerned with the problem of existence of periodic solutions for differential equations given in the following {\it standard form}:
 \begin{equation}\label{eq:e1}
	x^\prime = \e f(t,x,\e), \quad (t,x,\e)\in\R\times D\times[0,\e_0],
	\end{equation}
where $D$ is an open subset of $\R^n,$ $\e_0>0$, and $f\colon\R\times D\times[0,\e_0]\to \R^n$ is assumed to be $T$ periodic in the variable $t$. Further conditions will be assumed later on.

The averaging method is an important technique for investigating differential equations given in the standard form \eqref{eq:e1}. It has a long history, starting with the development of celestial mechanics by  Clairaut, Laplace, and Lagrange, which was later formalized by the works of Fatou, Krylov, Bogoliubov, and Mitropolsky  \cite{BM,Bo,Fa,BK} (for a historical review, see  \cite[Chapter 6]{MN} and \cite[Appendix A]{sanders2007averaging}). The averaging method concerns about asymptotic estimates for solutions of differential equations. More specifically, it provides conditions under which the solutions of  \eqref{eq:e1} remains $\e$-close to the solutions of the {\it truncated average equation} $x'=\e f_1(x),$ for an order $\e^{-1}$ interval of time. Here, assuming that,
for each $z\in D,$ the map $t\in\R\mapsto f(t,z,0)$ is integrable and $T$-periodic, $f_1(z)$ denotes its average in the variable $t$, that is,
\[
f_1(z)=\dfrac{1}{T}\int_0^T f(t,z,0)d t.
\]

Periodic trajectories are of major importance in the understanding of the qualitative behavior of differential equations. The averaging method is one of the most powerful technique for investigating periodic solutions of differential equations given as \eqref{eq:e1}.
One can find many results in the classical research literature relating isolated zeros of the {\it averaged function} $f_1(z)$
with $T$-periodic solutions of \eqref{eq:e1} (see, for instance, \cite{Hale,sanders2007averaging,Ver06}).
 The aforementioned results assume smoothness of $f$. These results have been generalized in several directions: in \cite{buicua2004averaging,LliNovTei2014,NS21}, for nonsmooth continuous differential equations; in \cite{LliMerNov15,LliNovRod17,LlilNovTei15}, for several classes of discontinuous piecewise smooth differential equations; and, in \cite{CMZ21}, for abstract semilinear equations.

 In this paper, it is assumed that the differential equation \eqref{eq:e1} satisfies the following Carath\'{e}odory conditions:
\begin{itemize}
\item[{\bf A.1}] $(x,\e)\in D\times[0,\e_0]\mapsto f(t,x,\e)$ is continuous, for almost every $t\in[0,T]$;

\smallskip

\item[{\bf A.2}] $t\in[0,T]\mapsto f(t,x,\e)$ is measurable, for every $(x,\e)\in D\times[0,\e_0];$

\smallskip

\item[{\bf A.3}] for each $r>0$, there exists an integrable function $g_r:[0,T]\to\R$ such that, for almost every $t\in[0,T]$, $|f(t,x,\e)|\leq g_r(t)$ provided that $|(x,\e)|\leq r.$
\end{itemize}
Accordingly, \eqref{eq:e1} is called {\it Carath\'{e}odory differential equation} (see \cite{kur86}).

Notice that, as consequences of the dominate convergence theorem, conditions {\bf A.2} and {\bf A.3} allow the computation of the averaged function $f_1(z)$ which, taking conditions {\bf A.1} and {\bf A.3} into account, is continuous on $D.$

There are some results in the research literature concerning periodic solutions for Carath\'{e}odory differential equations. In \cite{mawhin79}, an abstract continuation result for operator equations \cite[Theorem IV.13]{mawhin79}, due to Mawhin \cite{mawhin69a,mawhin69b} (see also \cite[Corollary IV.1]{gaines1977coincidence}), was used to prove a {\it guiding function method} \cite[Theorem VI.2]{mawhin79} that provides sufficient conditions for the existence of periodic solutions for Carath\'{e}odory differential equations. In \cite{CapMawZan92}, by using an approximation procedure, \cite[Theorem 1]{CapMawZan92} established that the coincidence degree of an autonomous differential equation $x'-g(x)=0$ can be computed in terms of the Brouwer degree of $g$ and, then, it was used to provide a continuation result \cite[Theorem 2]{CapMawZan92} for periodic solutions of Carath\'{e}odory differential systems given by $x'=g(x)+e(t,x)$.  The above studies treated the problem of existence of periodic solutions by means of topological methods.

In this paper, an abstract continuation result for operator equations, \cite[Theorem IV.2]{gaines1977coincidence}, 
will be used to provide sufficient conditions on the averaged function $f_1(z)$ that guarantee the existence of periodic solutions for the Carath\'{e}odory differential equation \eqref{eq:e1}. Additional conditions will be also provided to ensure that  a periodic solution converges uniformly to a constant function as $\e$ goes to $0$. In what follows, $d_B$ denotes the Brouwer degree (the reader is referred to Section \ref{sec:contres} for its formal definition).

\begin{mtheorem}\label{main}
Consider the Carath\'{e}odory differential equation \eqref{eq:e1}. Assume that there exists an open bounded set $V\subset\R^n,$ with $\ov V\subset D,$ such that $f_1(z)\neq0,$ for every $z\in \p V,$ and $d_B(f_1,V,0)\neq 0.$ Then, there exists $\e_V>0$ such that, for each $\e\in[0,\e_V]$, the differential equation \eqref{eq:e1} has a $T-$periodic solution $\f(t,\e)$ satisfying $\f(t,\e)\in V,$ for every $t\in[0,T].$ In addition, if there exists $z^*\in V$ such that $f_1(z^*)=0$ and $f_1(z)\neq0$ for every $z\in \ov{V}\setminus\{z^*\},$ then $\f(\cdot,\e)\to z^*,$ uniformly, as $\e\to 0.$
\end{mtheorem}
	
\begin{remark}\label{rem}Under the assumptions ``$z\mapsto f_1(z) $ is differentiable in a neighbourhood of $z=z^*,$  $f_1(z^*)=0,$ and $f_1(z)\neq0$ for every $z\in\ov V\setminus\{z^*\}$'',   the condition $d_B(f_1,V,0)\neq0$ holds provided that  $\det(D f_1(z^*))\neq0$ and, in that case, $d_B(f_1,V,0)=\sgn(\det(D f_1(z^*)))$.
\end{remark}

Usually, the averaging method for studying periodic solutions, even for continuous differential equations, assumes a smooth dependence of $f$ on the parameter of perturbation $\e$. Also, the existence of periodic solutions strongly relies on the existence of isolated zeros of the averaged function $f_1$ (see, for instance, \cite[Theorem 1.2]{buicua2004averaging} and \cite[Corollary 2.2]{NS21}). Here, Theorem \ref{main}, in addition to extending the previous results for Carath\'{e}odory differential equations, relaxes such a smoothness assumption on $f$ and the existence of periodic solutions is provided regardless the existence of isolated zeros of $f_1$. Its proof, based on the degree theory, is similar to the ones performed in \cite{NS21}. Here, the dominate convergence theorem is needed for defining suitable operators in order to employ an abstract continuation result for operator equations. 

The degree theory and an abstract continuation result for operator equations are presented in Section \ref{sec:contres}. Theorem \ref{main} is, then, proven in Section \ref{sec:proof}. Section \ref{sec:exa} is devoted to investigate some examples.

\section{Continuation result for operator equations}\label{sec:contres}

This section presents a useful abstract continuation result, Theorem \ref{t.aux}, 
which provides sufficient conditions for the existence of solutions of the operator equation
\begin{align}\label{eq:opeq}
Lx = \e{N}(x,\e),\,\, x\in\ov\Omega,
\end{align}
where $ L:\dom L\subset X\to Y$ is a linear map between real normed vector spaces $ X $ and $ Y,$ $\dom L$ is a subspace of $X,$ $\e\in[0,\e_0],$ with $\e_0>0$, $\Omega$ is an open bounded subset of $ X$ such that $\ov \Omega\subset\dom L,$  and $ N:\ov \Omega\times[0,\e_0]\to Y $ is any map.

In addition, it is assumed that $ L $ is {\it Fredholm of index $0$} and that $N$ is   $ L- ${\it compact}. Following \cite[Chapter I]{mawhin79}, such concepts are introduced in the sequel.

The linear map $ L $ is said to be Fredholm of index $0$ if $\ker L$ has finite dimension, $\Im L$ is closed in $Y$ and has finite codimension (that is  $\coker L = Y/\Im L$ has finite dimension), and the dimension of $\ker L$ coincides with the codimension of $\Im L$.
 In this case, there exist continuous projections $ P:X\to X $ and $ Q:Y\to Y $ satisfying 
\begin{equation}\label{fred}
\Im P = \ker L\,\,\text{ and }\,\,\ker Q=\Im L,
\end{equation}
which implies that
\[
X=\ker L\oplus\ker P\,\,\text{ and }\,\, Y=\Im L\oplus \Im Q.
\]
Notice that $\Im Q$ is isomorphic to $\coker L$ and, consequently, isomorphic to $\ker L$. In addition, one can see that $ L_P = L\big|_{\ker P~\cap~\dom L} $ is an isomorphism onto $\Im L$. Accordingly, denote 
\[
K_P=L_P^{-1}:\Im L\to \ker P\cap\dom L
\] and define $K_{P,Q} = K_P (Id- Q). $ 
 
The map  $N:\ov \Omega\times[0,\e_0]\to Y $ is said to be {\it $ L- $compact} on $\ov \Omega\times[0,\e_0]$ if the maps $ K_{P,Q} N:\ov \Omega\times[0,\e_0]\to X$ and $Q N:\ov \Omega\times[0,\e_0]\to Y$ are
continuous and the subsets $K_{P,Q} N (\ov \Omega\times[0,\e_0]) $  and $Q N(\ov \Omega\times[0,\e_0]) $  are relatively compact on $X$ and $Y$, respectively. It can be seen that the definition of $ L- $compactness does not depend on the choices of $ P $ and $ Q$ satisfying \eqref{fred}.

\begin{theorem}[{\cite[Theorem IV.2]{gaines1977coincidence}}]\label{t.aux}
Let $L,$ $N,$ and $Q$ be like above and $J:\Im Q\to\ker L$ any isomorphism. Assume that the following conditions are verified:
\begin{itemize}

\item[{\bf H.1}]$ Q{N}(x,0)\neq 0, $ for every $ x\in\partial\Omega\cap\ker L;$ and

\item[{\bf H.2}]$ d_B(JQ{N}(\cdot,0)\big|_{\Omega\cap\ker L},\Omega\cap\ker L,0)\neq 0$.
\end{itemize}
Then, there exists $\e_1\in(0,\e_0]$ such that, for each $\e\in[0,\e_1]$, the operator equation \eqref{eq:opeq} admits a solution in $ \Omega$.
\end{theorem}

Theorem \ref{t.aux} makes use of the concept of Brouwer degree $d_B$, which is defined as follows  (see \cite{zeidler}). Let $ V\subset\R^n $ be an open bounded subset of $ \R^n, $ $ g:\ov{V}\to \R^n $ a continuous function, and $y_0\notin g(\p V).$ The Brouwer degree $d_B(g,V,y_0)$ is characterized as the unique integer-valued function satisfying the following properties:
\begin{itemize}
	\item[{\bf P.1}] If $ d_B(g,V,y_0)\neq 0, $ then $ y_0\in g(V). $ Furthermore, if $ \mathbb{1}:\ov{V}\to\R^n $ is the identity function and $ y_0\in V, $ then $ d_B(\mathbb{1}, V, y_0) = 1. $
	\item[{\bf P.2}] If $V_1, V_2\subset V$ are disjoint open subsets of $V$ such 
	that $y_0\notin g(\ov{V}\backslash (V_1\cup V_2)),$ then $$d_B(g, V, y_0) = 
	d_B(g\big|_{V_1}, V_1, y_0) + d_B(g\big|_ {V_2}, V_2, y_0).$$
	\item[{\bf P.3}] If $ \{g_{\sigma}:\ov{V}\to\R^n\, |\, \sigma\in [0,1] \} $ is a continuous homotopy and $ \{ y_{\sigma}\, |\, \sigma\in [0,1] \} $ is a continuous curve in $Y$ such that $ y_{\sigma}\notin g_{\sigma}(\p V),\,\forall \sigma\in [0, 1] $ then $ d_B(g_{\sigma},V,y_{\sigma}) $ is constant for $ \sigma\in[0,1]. $
\end{itemize}

\section{Proof of Theorem \ref{main}}\label{sec:proof}
The proof of Theorem \ref{main} will follow directly from Propositions \ref{prop1} and \ref{prop2}.

\begin{proposition}\label{prop1}
Consider the Carath\'{e}odory differential equation \eqref{eq:e1}. Assume that there exists an open bounded set $V\subset\R^n,$ with $\ov V\subset D,$ such that $f_1(z)\neq0,$ for every $z\in \p V,$ and $d_B(f_1,V,0)\neq 0.$ Then, there exists $\e_V>0$ such that, for each $\e\in[0, \e_V]$, the differential equation \eqref{eq:e1} has a $T-$periodic solution $\f(t,\e)$ satisfying $\f(t,\e)\in V,$ for every $t\in[0,T].$
\end{proposition}
\begin{proof}
The first step of this proof consists in establishing suitable spaces and operators for applying Theorem \ref{t.aux}. It is worth mentioning that the initial framework of this proof is similar to the proof of \cite[Theorem A]{NS21}.

Consider the following real Banach spaces \[  X= \{ x\in C([0, T],\R^n)\colon~ x(0) = 
x(T)\} \,\mbox{ and }\, Y = \{ x\in C([0, T],\R^n)\colon~ x(0) = 0\},\]
and let $\Omega$ be the following open bounded subset of $ X,$
\[
\Omega = \{ x\in X\colon x(t)\in V,\,\forall\, t\in [0, T] \}.
\] 
As usual, $C([0, T],\R^n)$ denotes the space of continuous functions $x:[0,T]\to\R^n$ endowed with the sup-norm.

Define the linear map $ L\colon X\to Y $ by
\[
Lx(t) = x(t)- x(0)
\]
and let $ N:\ov\Omega\times[0,\e_0]\to Y $ be given by $$ N(x,\e)(t) = \int_0^t \, f(s,x(s), \e)d s.$$

In order to see that $N$ is well defined, let $r_V>0$ satisfy $|(x(t),\e)|<r_V$ for every $(t,x,\e)\in[0,T]\times \ov \Omega\times [0,\e_0].$ From condition {\bf A.3}, there exists an integrable function $g_{r_V}:[0,T]\to\R$ satisfying $|f(t,x,\e)|\leq g_{r_V}(t)$ for every $(x,\e)\in\ov V\times[0,\e_0]$ and for almost every $t\in[0,T].$ Thus, taking conditions {\bf A.2} and {\bf A.3} into account, it follows that the integral $N(x,\e)(t)$ is well defined for each $(x,\e)\in\ov\Omega\times[0,\e_0]$ and for every $t\in[0,T]$. In addition, consider the function $G_{r_V}:[0,T]\to\R$ given by
\[
G_{r_V}(t)=\int_0^t g_{r_V}(s)ds.
\]
From the dominate convergence theorem, $G_{r_V}$ is continuous on $[0,T]$ and, therefore, uniformly continuous. In addition, one can see that
\begin{equation}\label{rel1}
|N(x,\e)(t)|\leq G_{r_V}(t),
\end{equation}
for every $t\in[0,T]$, and
\begin{equation}\label{rel2}
|N(x,\e)(t)-N(x,\e)(\tau)|\leq |G_{r_V}(t)-G_{r_V}(\tau)|,
\end{equation}
for every $t,\tau\in[0,T]$. The relationship \eqref{rel1} implies that $N(x,\e)(0)=0$. The relationship \eqref{rel2} and the uniform continuity of $G_{r_V}$ imply that $t\mapsto N(x,\e)(t)$ is continuous. Hence, $N(x,\e)\in Y$ for every $(x,\e)\in\ov\Omega\times[0,\e_0]$, which implies that $N$ is well defined.

Now, given $(x^*,\e^*)\in\ov\Omega\times[0,\e_0],$ consider a sequence $(x_n,\e_n)\in\ov\Omega\times[0,\e_0]$ such that $(x_n,\e_n)\to(x^*,\e^*)$. For $t\in[0,T]$, define \[
\Delta_n(t)=f(t,x^*(t),\e^*)-f(t,x_n(t),\e_n).
\]
Notice that
\[
|N(x^*,\e^*)-N(x_n,\e_n)|\leq\int_0^T|\Delta_n(s)|ds.
\]
Since, for almost every $t\in[0,T]$, $\lim\Delta_n(t)=0$ (by {\bf A.1}) and $|\Delta_n(t)|\leq 2g_{r_V}(t)$ (by {\bf A.3}), again from the dominate convergence theorem, one has that
\[
\lim |N(x^*,\e^*)-N(x_n,\e_n)|=\lim \int_0^T|\Delta_n(s)|ds=0,
\]
which implies that $N$ is continuous.

\smallskip

Now, consider the operator equation
\begin{equation}\label{eq:op1}
Lx =\e N(x,\e), \,\, x\in\ov\Omega.
\end{equation}
Notice that $x\in \ov\Omega$ is a solution of \eqref{eq:op1} if, and only if, it can can be continued to a $ T-$periodic solution of the differential equation \eqref{eq:e1} in $\ov V.$

In what follows, the conditions of Theorem \ref{t.aux} will be verified for the operator equation \eqref{eq:op1}.

First, notice that  $\Im L=X\cap Y$ is closed in $Y$ and
\[
\ker L = \big\{ x\in X:\, x(t) = z\in\R^n,\,\forall\, t\in[0,T]\big\}
\] 
 is isomorphic to $ \R^n.$  It is easy to see that $\coker L$ is also isomorphic to $\R^n$. Hence, $L$ is Fredholm of index $0$.

In order to see that $ N $ is $ L- $compact on $\ov\Omega\times [0,\e_0]$,  let $ P:X\to X $ and $ Q:Y\to Y $ be continuous projections given, respectively, by
\[
 Px(t) = x(0)\,\, \text{ and }\,\, Qy(t) = t\dfrac{y(T)}{T}, \,\, \text{ for }\,\, t\in [0,T].
\]
Notice that Property \eqref{fred} is satisfied for $P$ and $Q$. Clearly, $K_{P,Q}N$ and $QN$ are continuous. Moreover, taking into account \eqref{rel1}, \eqref{rel2}, and the Arzel\`{a}-Ascoli Theorem, one can see that the families of continuous functions $K_{P,Q}N(\ov \Omega\times[0,\e_0])$ and $QN(\ov\Omega \times[0,\e_0])$  are relatively compact on $X$ and $Y$, respectively. Hence, $ N$ is $ L- $compact on $\ov\Omega\times[0,\e_0]$.

 		Now, let $ x\in\partial\Omega\cap\ker L.$ Notice that $ x(t)\equiv z\in\partial V.$ Thus, 
				\begin{equation}\label{QN}
		QN(x,0)(t) = \dfrac{t}{T}\int_0^T f(s,z,0)d s =t f_1(z).
		\end{equation}
		By hypothesis, $f_1(z)\neq0$ for every $z\in\p V,$ so $QN(x,0)\neq0$ for every  for $ x\in\partial\Omega\cap \ker L.$ 
		  Therefore,  condition {\bf H.1}  of Theorem \ref{t.aux}  holds for the operator equation \eqref{eq:op1}. 
		  
		  In addition, let the isomorphism $J:\Im Q\to\ker L$  be given by $$J\,y(t)=\dfrac{y(T)}{T}.$$
From \eqref{QN},  $JQN(x,0)=f_1(z),$ for every $x\in\Omega\cap\ker L$ (in this case, $x(t)\equiv z\in V$).  Thus,
		\[ d_B(JQN(\cdot,0)\big|_{\Omega\cap \ker L },\Omega\cap\ker L ,0) = d_B(f_1,V ,0). \]
		 By hypothesis, $d_B(f_1,V ,0)\neq0,$ so condition {\bf H.2} holds  for the operator equation \eqref{eq:op1}.
		 
		 Thus, Theorem \ref{t.aux} provides the existence of $\e_V\in(0,\e_0]$ such that, for each $\e\in[0,\e_V]$, the operator equation \eqref{eq:op1} has a solution $x_{\e}\in \Omega$. Hence, for each $\e\in[0,\e_{V }],$ $\f(t,\e)=x_{\e}(t)$ is a $T-$periodic solution of the Carath\'{e}odory differential equation \eqref{eq:e1} satisfying  $\f(t,\e)\in V $ for every $t\in[0,T].$  
		 \end{proof}

\begin{proposition}\label{prop2}
In addition to  hypotheses of Proposition \ref{prop1}, assume that there exists $z^*\in V$ such that $f_1(z^*)=0$ and $f_1(z)\neq0$ for every $z\in \ov{V}\setminus\{z^*\}.$ Then, for each $\e\in[0, \e_V]$, the differential equation \eqref{eq:e1} has a $T-$periodic solution $\f(t,\e)$ satisfying $\f(\cdot,\e)\to z^*$, uniformly, as $\e\to 0.$
\end{proposition}
\begin{proof} 

Let $\mu_0>0$ satisfy $V_\mu: = B(z^*,\mu)\subset V$ for every $\mu\in(0,\mu_0]$. 

By hypothesis, $f_1(z)\neq0,$ for every $x\in \p V_{\mu}$ and $\mu\in(0,\mu_0].$  

Also,  Properties {\bf P.1} and {\bf P.2} of the Brouwer degree imply that 
\[
d_B(f_1,V_{\mu},0)=d_B(f_1,V,0)\neq 0,
\]
 for every $\mu\in(0,\mu_0]$. Indeed, consider $V'=\inte(V\setminus V_{\mu})$, since $0\notin f_1(V')$, Property {\bf P.1} implies that $d_B(f_1,V',0)=0$. Now, since $0\notin f_1\big(\ov V\setminus(V_\mu\cup V')\big),$  Property {\bf P.2} implies that $$d_B(f_1,V,0)=d_B(f_1,V_{\mu},0)+d_B(f_1,V',0)=d_B(f_1,V_{\mu},0).$$

Hence, all the hypotheses of Proposition \ref{prop1} are satisfied by taking $V_{\mu}$ instead of $V$. Consequently, for each $\mu\in(0,\mu_0]$, 
there exists $\e_{V_{\mu}}>0$ such that, for each $\e\in[0, \e_{V_{\mu}}]$, the differential equation \eqref{eq:e1} has a $T-$periodic solution $\f^{\mu}(t,\e)$ satisfying $\f^{\mu}(t,\e)\in  V_{\mu},$ for every $t\in[0,T].$

Finally, a $T$-periodic solution $\f(t,\e)$ of the differential equation \eqref{eq:e1} converging uniformly to the constant function  $z^*$ will be obtained from the family of periodic solutions $\{\f^{\mu}:\R\times [0,\e_{V_{\mu}}]\rightarrow\R^n,\,\mu\in(0,\mu_0]\}$ obtained above. 

Let $\ov n\in\N$ satisfy $1/\ov n<\mu_0$. For each $n\in\N$, set $\mu_n=1/(\ov n+n)\in (0,\mu_0]$ and denote $\f_n=\f^{\mu_n}$ and $\e_n=\e_{V_{\mu_n}}\in(0,\e_0]$. Notice that, for each $n\in\N,$  $\f_n(t,\e)\in V_{\mu_n}=B(z^*,\mu_n)$ for every $t\in[0,T]$ and $\e\in[0,\e_n]$. Without loss of generality, one can assume that the sequence $\e_n$ converges to $\e^*\in[0,\e_0].$ From here, the cases $\e^*>0$ and $\e^*=0$ will be distinguished.

If $\e^*>0,$ then there exists $\ov\e\in(0,\e_0]$ and $n_0>0$ such that $[0,\ov\e]\subset[0,\e_n]$ for every $n>n_0$. Thus, the sequence of periodic solutions $\f_n(t,\e), n>n_0,$ satisfies $\f_n(t,\e)\in B(z^*,\mu_n)$ for every $t\in[0,T]$ and $\e\in[0,\ov \e]$. Thus, $\f_n\to z^*$ uniformly and, consequently, $\f(t,\e)=z^*$ is a periodic solutions of \eqref{eq:e1} for every $\e\in[0,\ov\e]$. 

Otherwise, if $\e^*=0,$ let $\e_M=\max\{\e_n:\,n\in\N\}.$ Then, for each $\e\in(0, \e_M]$ define
\[
n_{\e}=\max\{n\in\N:\,\e_n\geq\e\}.
\]
Notice that $n_{\e}\to \infty$ as $\e\to0$. Therefore, taking $\f(t,\e)=\f_{n_{\e}}(t,\e),$ one can see that $\f(t,\e)\in B(z^*,\mu_{n_{\e}}).$ Hence, $ \f(\cdot,\e)\to z^*,$ uniformly, as $\e\to 0.$
\end{proof}

\section{examples}\label{sec:exa}

Given $\e_0>0$ and $A=\{(x,\dot x)\in \R^2: r_0\leq \sqrt{x^2+\dot x^2}\leq r_1\}$, with $r_1>r_0>0,$ consider the following second order differential equation
\begin{equation}\label{eq:ex}
\ddot x=-x +\e\, g(x,\dot x,\e),\quad (x,\dot x,\e)\in A\times [0,\e_0].
\end{equation}
 Assume that $$\tilde g(\T,r,\e)=g(r \cos\T,r\sin\T,\e)$$ satisfies {\bf A.1}-{\bf A.3}.
Applying the polar change of coordinates, $(x,\dot x)=(r \cos\T,r\sin\T),$ and taking $ \theta $ as the new time variable, the second order differential equation \eqref{eq:ex} becomes
\begin{equation}\label{eq:trans}
\begin{array}{rl}
\dfrac{d r}{d\T}=\e f(\T,r,\e),
\end{array}
\end{equation}
where $(\T,r,\e)\in \R\times[r_0,r_1]\times[0,\e_0]$ and
\[
f(\T,r,\e)=-\dfrac{r\,\tilde g(\T,r,\e)\sin\T}{r-\e \tilde g(\T,r,\e)\cos\T}.
\]
Notice that, by taking $\e_0>0$ smaller if necessary, $f$ also satisfies {\bf A.1}-{\bf A.3}. In addition,
\[
f_1(r)=-\dfrac{1}{2\pi}\int_0^{2\pi} \tilde g(\T,r,0)\sin\T \, d\T.
\]

\begin{example}
Suppose that $[1,2]\subset(r_0,r_1)$ and let
\[
g(x,y,0)=\sgn\big(y(x^2+y^2-1)\big)\max\big\{0,(x^2+y^2-1)(x^2+y^2-4)\big\}.
\]
Notice that $$\tilde g(\T,r,0)=\sgn\big((r^2-1)\sin \T\big)\max\big\{0,(r^2-1)(r^2-4)\big\}$$
satisfies {\bf A.1}-{\bf A.3} and
\[
f_1(r)=\begin{cases}
\dfrac{2}{\pi}(r^4-5 r^2+4),& r<1,\\
0,& 1\leq r\leq2,\\
-\dfrac{2}{\pi}(r^4-5 r^2+4),& r>2.
\end{cases}
\]
Now, let $V$ be any open interval in $[r_0,r_1]$ containing $[1,2]$. It is easy to construct a homotopy $\{g_{\sigma}:\ov{V}\to\R^n\, |\, \sigma\in [0,1] \} $ between $f_1\big|_{\ov V}$ and $g_1:r\mapsto 3/2-r,\,r\in \ov V,$ such that $ 0\notin g_{\sigma}(\p V),\,\forall \sigma\in [0, 1]$ (see Figure \ref{fig1}). 
\begin{figure}[H]
	\begin{center}
		\begin{overpic}[width=10.5cm]{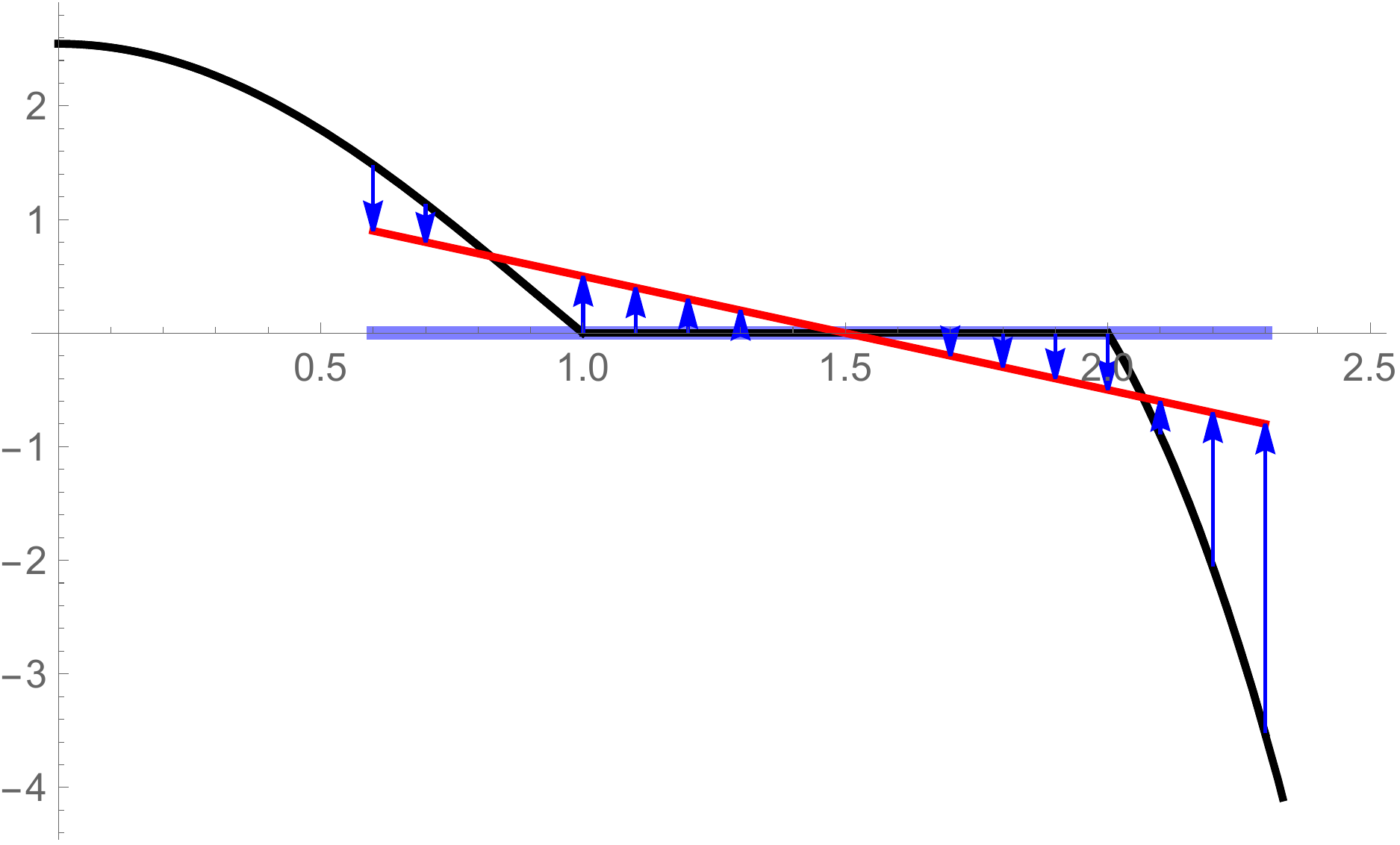}
		\put(92,3){$f_1(r)$}
		\put(92,29){$g_1(r)$}
		\put(101,36){$r$}
		\put(88,40){$V$}
		\end{overpic}
		\caption{Illustration of the homotopy $g_{\sigma}$ between $f_1(r)$ and $g_1(r)=3/2-r$. Notice that  $ 0\notin g_{\sigma}(\p V),\,\forall \sigma\in [0, 1]$.}
	\label{fig1}
	\end{center}
	\end{figure}
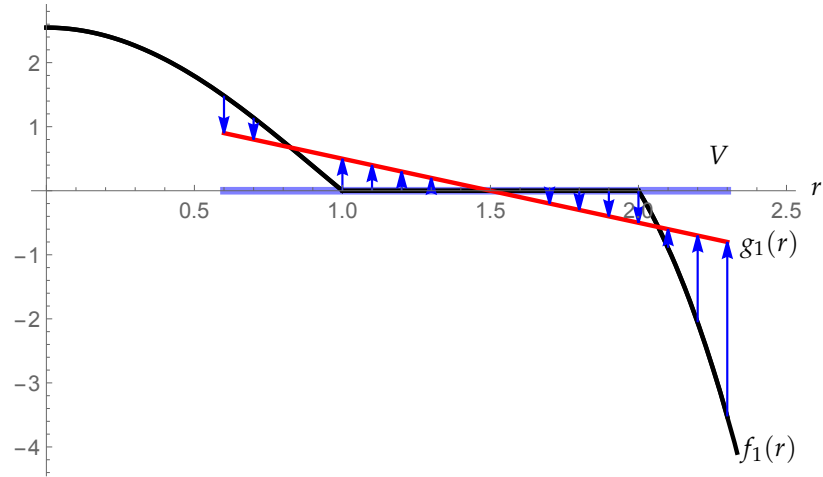

\noindent Property {\bf P.3} of the Brouwer degree implies $$ d_B(f_1,V,0)=d_B(g_1,V,0)=-1$$ (see Remark \ref{rem}). Hence, Theorem \ref{main} guarantees, for $\e>0$ sufficiently small, the existence of a periodic solution $r(\T,\e)$ of the differential equation \eqref{eq:trans} satisfying $r(\T,\e)\in V$ for every $\T\in\R$. This corresponds to a periodic solution $(x(t,\e),\dot x(t,\e))$ of the second order differential equation \eqref{eq:ex} satisfying $(x(t,\e),\dot x(t,\e))\in A_V$ for every $t\in\R,$ where $A_V=\{(x,\dot x)\in \R^2: \sqrt{x^2+\dot x^2}\in V\}.$
\end{example}

\begin{example}
Suppose that $r_0<1<r_1$ and let
\[
g(x,y,0)=\sgn(y)\sqrt[3]{x^2+y^2-1}.
\]
Notice that $$\tilde g(\T,r,0)=\sgn(\sin \T)\sqrt[3]{r^2-1}$$ satisfies {\bf A.1}-{\bf A.3} and
\[
f_1(r)=-\dfrac{2}{\pi}\sqrt[3]{r^2-1}.
\]
Moreover, $r^*=1$ is the unique zero of $f_1$. Thus, let $V$ be any open interval in $[r_0,r_1]$ containing $r^*$. Again, it is easy to construct a homotopy $\{g_{\sigma}:\ov{V}\to\R^n\, |\, \sigma\in [0,1] \} $ between $f_1\big|_{\ov V}$ and $g_1:r\mapsto 1-r,\,r\in \ov V,$ such that $ 0\notin g_{\sigma}(\p V),\,\forall \sigma\in [0, 1]$ (see Figure \ref{fig2}). 

	\begin{figure}[H]
	\begin{center}
		\begin{overpic}[width=10.5cm]{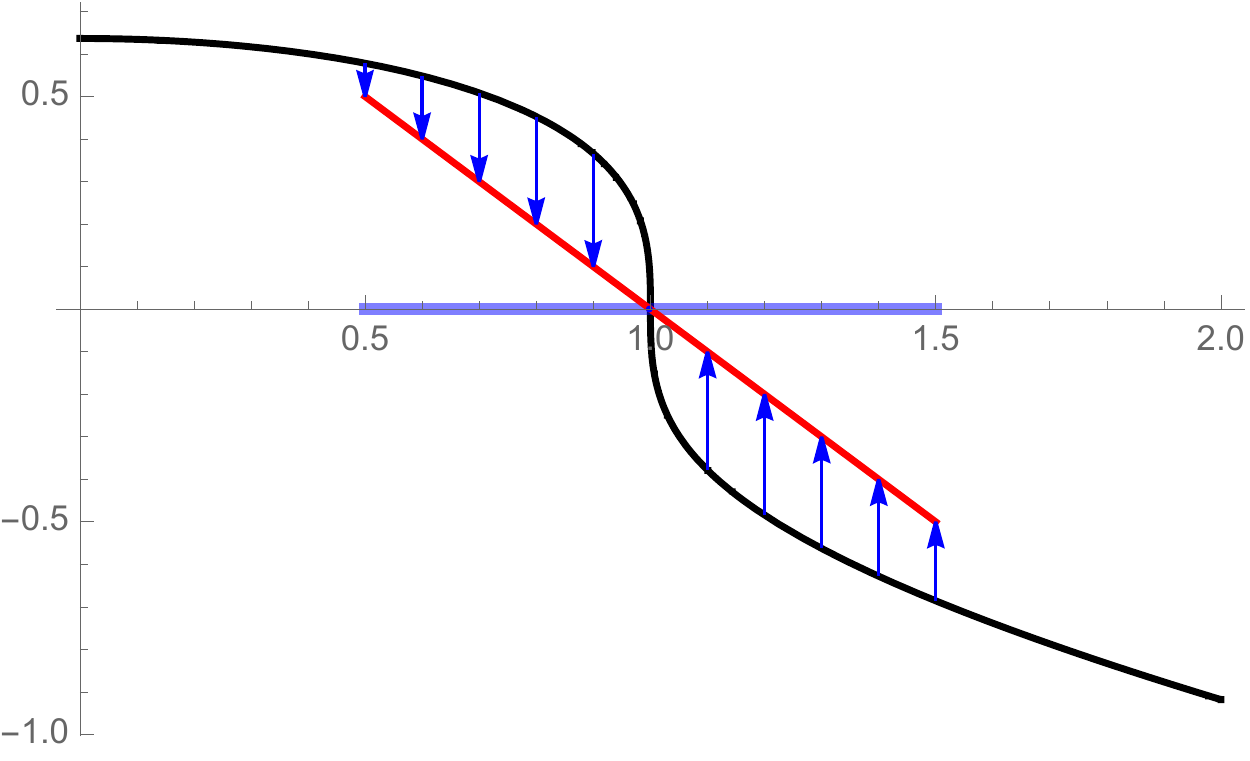}
		\put(94,9){$f_1(r)$}
		\put(76,19){$g_1(r)$}
		\put(102,35.5){$r$}
		\put(73,38){$V$}
		\end{overpic}
\caption{Illustration of the homotopy $g_{\sigma}$ between $f_1(r)$ and $g_1(r)=1-r$.}
	\label{fig2}
	\end{center}
	\end{figure}
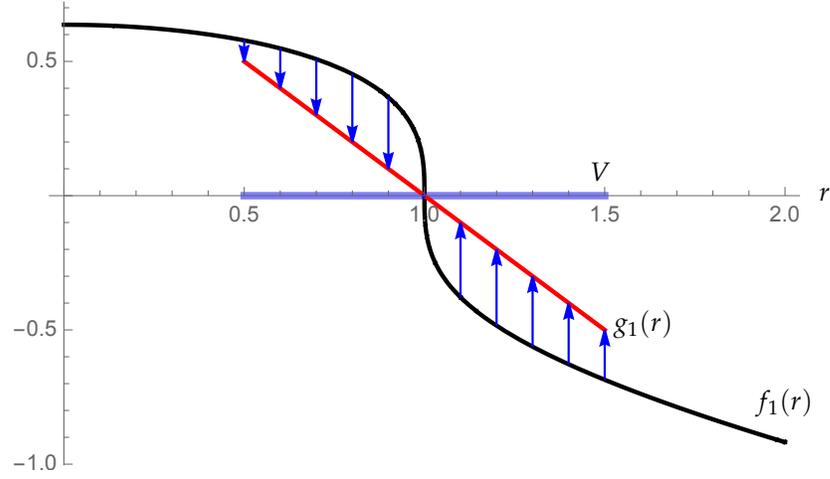
\end{example}

\noindent Again, Property {\bf P.3} of the Brouwer degree implies $$ d_B(f_1,V,0)=d_B(g_1,V,0)=-1$$ (see Remark \ref{rem}). Hence, Theorem \ref{main} guarantees, for $\e>0$ sufficiently small, the existence of a periodic solution $r(\T,\e)$ of the differential equation \eqref{eq:trans} satisfying $r(\cdot,\e)\to r^*$, uniformly, as $\e\to0$. This corresponds to a periodic solution $(x(t,\e),\dot x(t,\e))$ of the second order differential equation \eqref{eq:ex} satisfying $|(x(\cdot,\e),\dot x(\cdot,\e))|\to r^*$, uniformly, as $\e\to0$.

\begin{remark}
 The procedure employed in the examples above for computing explicitly the Brouwer degree of $(f_1,V,0)$ always works for one-dimensional functions under similar suitable conditions, namely, if $a,b\in\R$, $a<b$, $V = [a,b],$ and $f : V \to \R$ is continuous with $f(a)f(b) < 0$, then $d_B(f,V,0) = \sgn(f(b)-f(a))$.
\end{remark}

\section*{Acknowledgements}

The author thanks the referee for the constructive comments and suggestions which led to an improved version of the paper.

DDN is partially supported by S\~{a}o Paulo Research Foundation (FAPESP) grants 2018/16430-8, 2018/ 13481-0, and 2019/10269-3, and by Conselho Nacional de Desenvolvimento Cient\'{i}fico e Tecnol\'{o}gico (CNPq) grants 306649/2018-7 and 438975/ 2018-9.

\bibliographystyle{abbrv}
\bibliography{references.bib}
\end{document}